\documentclass[11pt, oneside]{amsart}

\usepackage[english]{babel} 
\usepackage{graphics,color,pgf}
\usepackage{epsfig}
\usepackage[ansinew]{inputenc}
\usepackage[all]{xy}
\usepackage{hyperref}
\newdir{ >}{!/8pt/@{}*@{>}}
\usepackage{amssymb, amsmath,amsthm, mathtools, amscd}
\usepackage{mathrsfs}
\usepackage[margin=1.5in]{geometry}

\makeatletter
\newtheorem*{rep@theorem}{\rep@title}
\newcommand{\newreptheorem}[2]{%
\newenvironment{rep#1}[1]{%
 \def\rep@title{#2 \ref{##1}}%
 \begin{rep@theorem}}%
 {\end{rep@theorem}}}
\makeatother

\newtheorem{intro_thm}{Theorem}

\newtheorem{lemma}{Lemma}[section]
\newtheorem{thm}[lemma]{Theorem} 
\newtheorem{prop}[lemma]{Proposition}
\newtheorem{cor}[lemma]{Corollary}

\theoremstyle{definition}
\newtheorem{defn}[lemma]{Definition}
\newtheorem{es}[lemma]{Example}

\theoremstyle{remark}
\newtheorem{oss}[lemma]{Remark}

\newtheoremstyle{TheoremNum}
        {\topsep}{\topsep}              
        {\itshape}                      
        {}                              
        {}                     
        {.}                             
        { }                             
        {\thmname{\bfseries #1}\thmnote{ \bfseries #3}}
    \theoremstyle{TheoremNum}
\newtheorem{rec_thm}{Theorem}

\newcommand\matR{{\mathbb{R}}}

\newcommand\matZ{{\mathbb{Z}}}
\newcommand\matC{{\mathbb{C}}}

\newcommand\matF{{\mathbb{F}}}

\newcommand\calS{{\mathcal S}}

\newcommand\calT{{\mathcal T}}

\newcommand\calB{{\mathcal B}}

\newcommand\calY{{\mathcal Y}}

\newcommand\calR{\mathcal{R}}

\newcommand\calX{\mathcal{X}}
\newcommand\calD{{\mathcal D}}
\newcommand\calO{{\mathcal O}}

\newcommand{\Hm}{\textup{\textup{H}}}
\newcommand{\Hb}{\textup{\textup{H}}_{\textup{b}}}
\newcommand{\Hc}{\textup{\textup{H}}_{\textup{c}}}
\newcommand{\Hcb}{\textup{\textup{H}}_{\textup{cb}}}

\newcommand{\Hcbb}{\textup{\textup{H}}_{\textup{c(b)}}}

\newcommand{\Linf}{\text{L}^{\infty}}

\newcommand{\rk}{\text{rk}}

\newcommand{\Isom}{\text{Isom}}

\newcommand{\Stab}{\text{Stab}}

\newcommand{\argu}{\text{arg}}

\newcommand\gencoc{{\sigma: \Gamma\times X\rightarrow G}}

\begin{document}

\title[Parametrized K\"{a}hler class]{Parametrized K\"{a}hler class and Zariski dense orbital 1-cohomology}

\author[F. Sarti]{F. Sarti}
\address{Department of Mathematics Giuseppe Peano, Universit\`{a} di Torino, Via Carlo Alberto 10, 10123 Torino, Italy}
\email{filippo.sarti@unito.it}

\author[A. Savini]{A. Savini}
\address{Section de Math\'ematiques, Universit\'{e} de Gen\`eva, Rue Du Conseil G\'eneral 7-9, Geneva 1205, Switzerland}
\email{alessio.savini@unige.ch}

\date{\today.\ \copyright{\ F. Sarti, A. Savini 2022}.
  The first author is supported through GNSAGA, funded by the INdAM, the second author is supported by the SNF grant n. 20020-192216.}

\begin{abstract}
Let $\Gamma$ be a finitely generated group and let $(X,\mu_X)$ be an ergodic standard Borel probability $\Gamma$-space. Suppose that $\calX$ is a Hermitian symmetric space not of tube type and assume that $G=\Isom(\calX)^{\circ}$ is simple. 
Given a Zariski dense measurable cocycle $\sigma:\Gamma \times X \rightarrow G$, we define the notion of parametrized K\"{a}hler class and we show that it completely determines the cocycle up to cohomology.
\end{abstract}
  
\maketitle


\section{Introduction}
Thanks to the pioneering work by both Ghys \cite{ghys:articolo} and Matsumoto \cite{matsumoto:articolo}, it is well-known that a circle action is completely determined by the pullback of the \emph{bounded Euler class} $e^b_{\mathbb{Z}} \in \Hb^2(\textup{Homeo}^+(\mathbb{S}^1);\mathbb{Z})$ along the representation which defines the action. It is natural to ask whether this rigidity phenomenon may appear in other different contexts. A suitable setting to answer to such a question is the one of Hermitian symmetric spaces. 

Given a Hermitian symmetric space $\calX$ with isometry group $G=\textup{Isom}(\calX)^\circ$, the existence of a $G$-invariant complex structure on $\calX$ allows to construct a natural K\"{a}hler form $\omega_{\calX}$.  Exploiting Dupont's isomorphism \cite{dupont76} between $G$-invariant differential forms on $\calX$ and the continuous cohomology of $G$, one argues that $\omega_{\calX}$ determines a non-vanishing class in degree two, the \emph{K\"{a}hler class} $k_G$. Remarkably such a class has a bounded analogue \cite{CO}, namely the \emph{bounded K\"{a}hler class} $k_G^b$. 

The study of the bounded K\"{a}hler class has interested many mathematicians so far. Just to mention few of them, we recall the work by Burger-Iozzi \cite{BIgeo}, by Burger, Iozzi and Wienhard \cite{BIW1} and by Pozzetti \cite{Pozzetti}, where the common denominator is the notion of maximality for representations, which is given by using the pullback of the bounded K\"{a}hler class. 
In particular, Pozzetti proved a rigidity result for representations of complex hyperbolic lattices into the group $\textup{SU}(m,n)$, while Burger, Iozzi and Wienhard gave a complete characterization of maximal representations from surface groups into a general Hermitian Lie group.

Going back to the similarities shared with the bounded Euler class, Burger, Iozzi and Wienhard \cite{BI04,BIW07} proved that the pullback of the bounded K\"{a}hler class determines uniquely the conjugacy class of any Zariski dense representation of a finitely generated group into a Hermitian group not of \emph{tube type}. A Hermitian group is of tube type if the associated symmetric space can be biholomorphically realized as $V+i \Omega$, where $V$ is a real vector space and $\Omega \subset V$ is a proper convex cone. 

The second author \cite{savini2021} has recently extended the result by Ghys to the context of measurable cocycles. The aim of this paper is to get a similar extension for the result by Burger, Iozzi and Wienhard \cite{BIW07}. Given a finitely generated group $\Gamma$ and a standard Borel probability $\Gamma$-space $(X,\mu_X)$, we are going to focus our attention on Zariski dense cocycles $\Gamma\times X\rightarrow G$ (Definition \ref{cocycle_definition}), where $G$ is a Hermitian group not of tube type. Here Zariski dense refers to the algebraic hull (Definition \ref{def:algebraic:hull}). 

Before stating the main results, we want to discuss shortly the techniques that we are going to use (most of them were introduced and discussed in \cite{BIW07}). Given a Hermitian symmetric space $\calX$ with $G=\text{Isom}(\calX)^\circ$, we can biholomorphically realize $\calX$ as a bounded domain $\calD_{\calX}$ in some complex vector space. In this case, even if the topological boundary $\partial \calD_{\calX}$ is not a $G$-homogeneous space, it contains a unique closed $G$-orbit $\calS_G$ called \emph{Shilov boundary} (Definition \ref{def_shilov_boundary}). The importance of such a boundary relies on the fact that one can define on it a preferred representative for bounded cohomology classes. In fact, setting
$$
\beta_{\textup{Berg}}:\calD^3_{\calX} \rightarrow \mathbb{R} \ , \ \ \beta_{\textup{Berg}}(x,y,z):=\int_{\Delta(x,y,z)} \omega_{\calX} \ ,
$$
where $\Delta(x,y,z)$ is any smooth triangle with geodesic sides, one can measurably extend $\beta_{\textup{Berg}}$ to a $G$-invariant Borel cocycle $\beta_G$ defined everywhere on the Shilov boundary $\calS_G$ \cite{BIW07}. The cocycle $\beta_G$ is a natural representative for the bounded K\"{a}hler class (Example \ref{es_implement_kahler_class}).

Moreover, the map $\beta_G$ satisfies the following identity \cite[Section 3.2]{BIW07}
$$
\langle\langle x,y,z\rangle\rangle = e^{i\beta_G(x,y,z)} \mod \ \mathbb{R}^\ast \ ,
$$
where the equality $\mod \mathbb{R}^\ast$ means that the two terms are equal up to a non-zero multiplicative real constant. The left-hand side of the above equation is the \emph{Hermitian triple product} of $x,y,z$ and it is defined in terms of the Bergman kernels on $\calD_{\calX}$ (Section \ref{sec:hermitian:spaces}).

The boundary $\calS_G$ can be identified with the quotient $(\mathbf{G}/\mathbf{Q})(\matR)$, where $\mathbf{G}$ is the connected adjoint $\matR$-group associated to the complexification of the Lie algebra of $G$ and $\mathbf{Q}<\mathbf{G}$ is a specific maximal parabolic subgroup \cite[Section 2.3]{BIW07}. Burger, Iozzi and Wienhard exploited such identification to extend the Hermitian triple product to a \emph{complex Hermitian triple product} $\langle\langle\cdot,\cdot,\cdot\rangle\rangle_{\matC}$ on $(\mathbf{G}/\mathbf{Q})^3$ \cite[Section 2.4]{BIW07}. They proved that the space $\calX$ is not of tube type if and only if the complex Hermitian triple product is not a constant function \cite[Theorem 1]{BIW07}. 

Another reason which justifies our interest in the existence of a measurable representative $\beta_G$ of $k_G^b$ relies on the role played by \emph{boundary maps} in the \emph{pull back} of cohomology classes along cocycles \cite{savini3:articolo,moraschini:savini,moraschini:savini:2}. Given a cocycle $\sigma:\Gamma\times X\rightarrow G$ as above, a boundary map $\phi:B\times X\rightarrow \calS_G$ is a measurable $\sigma$-equivariant map (Definition \ref{def_boundary_map}), where $B$ is a $\Gamma$-boundary (Definition \ref{def_boundary}). Boundary maps exist for instance for Zariski dense cocycles \cite[Theorem 1]{sarti:savini1} and allow to implement an alternative definition of the pull back along $\sigma$. 

In the setting described so far, we will prove the following
\begin{intro_thm}\label{main_theorem2}
Let $\Gamma$ be a finitely generated group, let $(X,\mu_X)$ be an ergodic standard Borel probability $\Gamma$-space and consider a Zariski dense measurable cocycle $\sigma:\Gamma\times X\rightarrow G$  into a simple Hermitian Lie group not of tube type.
Then the class $\Hb^2(\sigma)(k_G^b)$ in $\Hb^2(\Gamma;\textup{L}^\infty(X;\matR))$ is non-zero and it determines uniquely the cohomology class of $\sigma$.
\end{intro_thm}

If we denote by $\Hm^1_{ZD} (\Gamma \curvearrowright X;G)$ the space of equivalence classes of Zariski dense cocycles, Theorem \ref{main_theorem2} implies that we get an injection
\begin{equation}\label{equation_inclusion}
K_X: \Hm^1_{ZD} (\Gamma \curvearrowright X;G)\rightarrow \Hb^2(\Gamma;\Linf(X;\matR))\;,\;\;\; [\sigma]\mapsto \Hb^2(\sigma)(k_G^b).
\end{equation}

As \cite[Theorem 3]{BIW07} follows from the more general \cite[Theorem 4]{BIW07}, the same thing will happen in our case. Given two measurable cocycles $\sigma_1,\sigma_2:\Gamma \times X\rightarrow G_i$ where 
$G_i=\Isom (\calX_i)^{\circ}$ for $i=1,2$, we say that $\sigma_1$ and $\sigma_2$ are \emph{equivalent} if there exists an isomorphism $s:G_1 \rightarrow G_2$ such that $s \circ \sigma_1$ and $\sigma_2$ are cohomologous (Definition \ref{def_equivalent_cocycles}). Theorem \ref{main_theorem2} is a consequence of the following

 \begin{intro_thm}\label{main_theorem}
Let $\Gamma$ be a finitely generated discrete group and let $(X,\mu_X)$ be an ergodic standard Borel probability $\Gamma$-space. Let $\{\sigma_i:\Gamma\times X\rightarrow G_i\},i=1,\ldots,n$ be a family of Zariski dense measurable cocycles into simple Hermitian Lie groups not of tube type. If the cocycles are pairwise inequivalent, then the subset 
  $$\{\Hb^2(\sigma_i)(k^b_{G_i}), 1\leq i\leq n\}\subset \Hb^2(\Gamma;\textup{L}^\infty(X;\matR))$$
  is linearly independent over $\textup{L}^\infty(X;\matZ)$.
 \end{intro_thm}

The structure of the proof shares some similarities with the one of \cite[Theorem 4]{BIW07}. A key point of both their proofs and ours is the characterization of Hermitian spaces not of tube type in terms of the Hermitian triple product. However, in the measurable setting one has to overcome some new difficulties. Boundary maps for cocycles, whose existence is ensured by \cite[Theorem 1]{sarti:savini1}, must be studied carefully. In fact, for any $i=1,\ldots,n$ and $\sigma_i$-equivariant map $\phi_i:B\times X\rightarrow \calS_{G_i}$, we prove that the slices, namely the maps obtained by fixing $x\in X$, preserves transversality (see Section \ref{sec:hermitian:spaces} and Proposition \ref{prop_transverse}).
This means that for almost every $x\in X$ and almost every pairs of points $(b_1,b_2)\in B^2$, the points $\phi_i(b_1,x)$ and $\phi_i(b_2,x)$ are $G_i$-conjugated in $\calS_{G_i}$. 
Such property, combined with a classical Fubini argument, allows to twist the cocycles $\sigma_i$'s and their boundary maps $\phi_i$'s so that the image under $\phi_i$ of almost every pair of points $(b_1,b_2)\in B^2$ coincides with a fixed pair $\phi_i(b_1,x_0),\phi_i(b_2,x_0)$ for some $x_0\in X$.
Then, we show that a linear dependence between the pullback classes would imply that the essential image of almost every slice is contained in a proper Zariski closed subset of the Shilov boundary.
This leads to a contradiction with the Zariski density of the $\sigma_i$'s. 

As we will recall in Section \ref{sec_measurable_cocycles}, the set of measurable cocycles inherits a natural cohomological interpretation coming from the theory by Feldmann and Moore  about the cohomology of equivalence relations \cite{feldman:moore}. 
The injective map of Equation \eqref{equation_inclusion} shows how Theorem \ref{main_theorem2} realizes an inclusion of the subset of Zariski dense cocycles in the second bounded cohomology of $\Gamma$ with coefficients in $\Linf(X;\matR)$.
This allows to give some triviality statements about the space $\Hm^1_{ZD} (\Gamma \curvearrowright X;G)$ in some specific cases.
For instance if $\Gamma$ is a higher rank lattice whose bounded cohomology vanishes in degree two, we show that $\Hm^1_{ZD} (\Gamma \curvearrowright X;G)$ is trivial (see Proposition \ref{cor_high_rank}).This can be interpreted as an alternative approach to Zimmer' Superrigidity, since we have that under hypothesis of Zariski density, cocycles coincides with representations and there are no representations in this context.

Finally, we focus on the case of products, namely when $\Gamma$ is a lattice in a product $H=\prod\limits_{i=1}^n H_i$ whose projection on each factor is dense.  In Proposition \ref{proposition_product}, under certain suitable hypothesis on the $H$-action on the Borel probability space $X$, we show that the vanishing of the second bounded cohomology of each factor $H_i$ implies that cocycles $\Gamma\times X\rightarrow G$ into a Hermitian group not of tube type cannot exist.

\vspace{0.5 cm}

\paragraph{\textbf{Plan of the paper.}}
The paper is divided into three main sections. In Section \ref{section_preliminaries} our aim is to set the necessary theory and tools.  Section \ref{sec:hermitian:spaces} is devoted to recall the basics about Hermitian symmetric spaces, the notions of Shilov boundary, Bergman cocycle and Hermitian triple product. In Section \ref{sec_measurable_cocycles} we introduce measurable cocycles and we show how such objects inherits a cohomological interpretation. We then move to Section \ref{section_boundary_maps}, where we deal with the notions of boundary and boundary map. We conclude this introductory part giving a short overview about the theory of continuous bounded cohomology. We describe the pullback of cohomology classes along cocycles, which allows to define the main character of the paper, the parametrized K\"{a}hler class (Section \ref{section_bounded_cohomology}). 

In Section \ref{section_proof} we give the proofs of our main theorems and we conclude, in Section \ref{section_consequences}, with some applications following from our results. 


\vspace{ 0.5 cm}

\section{Preliminaries}\label{section_preliminaries}

\subsection{Hermitian symmetric spaces and K\"{a}hler class}\label{sec:hermitian:spaces}
In this section we deal with Hermitian symmetric spaces. We will distinguish them into spaces of tube type and not of tube type. Then we define the Shilov boundary of the bounded domain realization in $\matC^n$. The K\"{a}hler structure of a Hermitian symmetric space allows us to introduce the Bergman cocycle. We conclude recalling the Hermitian triple product and the complex Hermitian triple product. For the details about this section we refer either to \cite{BI04,BIW07} or to the book chapter \cite{Kor00}.

\begin{defn}
A symmetric space $\mathcal{X}$ with isometries $G=\Isom(\mathcal{X})^{\circ}$ is \emph{Hermitian} if it admits a $G$-invariant complex structure.  Given a semisimple algebraic group $\mathbf{G}$ defined over $\mathbb{R}$, we say that $G=\mathbf{G}(\matR)^{\circ}$ is of \emph{Hermitian type} if its associated symmetric space $\mathcal{X}$ is Hermitian. The notation $\mathbf{G}(\mathbb{R})^\circ$ refers to the connected component of the neutral element of the real points of $\mathbf{G}$.
\end{defn}

Hermitian symmetric spaces can be distinguished into spaces of \emph{tube type} and \emph{not of tube type}. The former ones are biholomorphic to a set of the form $V+i\Omega$, where $V$ is a real vector space and $\Omega\subset V$ is a proper convex cone. For instance, the group $\textup{SU}(p,q)$ of complex matrices preserving the Hermitian form $h_{p,q}$ of signature $(p,q)$ on $\mathbb{C}^{p+q}$ is of tube type only when $p=q$. In that case, the Hermitian symmetric space is biholomorphic to $\textup{Herm}(p,\mathbb{C}) + i \textup{Herm}^+(p,\mathbb{C})$, where $\textup{Herm}(p,\mathbb{C})$ is the space of complex Hermitian matrices and $\textup{Herm}^+(p,\mathbb{C})$ is the proper cone of positive definite ones (for $p=1$ we get back the upper half-space realization of the hyperbolic plane $\mathbb{H}^2_{\mathbb{R}}$). 

For both spaces of tube type and not of tube type, it is a standard fact that there exists a \emph{bounded domain realization} (the Harish-Chandra realization, \cite[Theorem.III.2.6]{Kor00}), that is a biholomorphism between $\mathcal{X}$ and a bounded connected open subset $\mathcal{D}_{\calX}\subset \matC^n$. The group $G$ acts on $\mathcal{D}_{\calX}$ via biholomorphisms and such an action can be continuously extended to the topological boundary $\partial \mathcal{D}_{\calX}$. 

\begin{defn}\label{def_shilov_boundary}
The \emph{Shilov boundary} of a bounded domain $\mathcal{D}\subset\matC^n$ is the unique minimal closed subset $\calS$ of $\partial\mathcal{D}$ such that, for any continuous function $f$ on $\overline{\mathcal{D}}$ which is holomorphic on $\mathcal{D}$, it holds
$$|f(z)|\leq \max\limits_{y\in\calS}|f(y)|$$
for every $z\in\mathcal{D}$.
\end{defn}

In the Harish-Chandra realization the Shilov boundary $\calS_G$ can be identified with the $G$-orbit of a unique point \cite[Section 2.3]{BIW07}, so $\calS_G$ is a homogeneous $G$-space. When $G=\textup{SU}(p,q)$, the associated Shilov boundary $\calS_{p,q}$ parametrizes totally isotropic subspaces of $\mathbb{C}^{p+q}$ with respect to the form $h_{p,q}$ having maximal dimension $d= \min \{ p, q \}$. It is well-known that $\calS_{p,q}$ is a homogeneous $G$-space and it can be identified with the quotient $G/Q$, where $Q$ is the stabilizer of a fixed isotropic subspace of dimension $d$. The latter identification can be actually extended to any Hermitian symmetric space. More precisely, we consider a connected real algebraic group $\mathbf{G}$ corresponding to the complexification of a Lie group of Hermitian type $G=\mathbf{G}(\mathbb{R})^\circ$. By \cite[Section 2.3.1]{BIW07} there exists a \emph{maximal proper parabolic subgroup} $\mathbf{Q}<\mathbf{G}$, such that $\calS_G$ is isomorphic to $(\mathbf{G}/\mathbf{Q})(\mathbb{R})=G/Q$, where $Q=G \cap \mathbf{Q}$.
In this way $\calS_G$ is realized as the real points of the complex projective variety $\mathbf{G}/\mathbf{Q}$.

The product $\calS_G \times \calS_G$ contains a unique open $G$-orbit $\calS_G^{(2)}$ whose elements are pairs of \emph{transverse} points (in particular, $G$ acts transitively on transverse pairs).  In the example of $G=\textup{SU}(p,q)$, two maximal isotropic subspaces $V,W \subset \mathbb{C}^{p+q}$ are transverse if their intersection is trivial, that is $V \cap W= \{ 0 \}$. Given a point $\xi \in \calS_G$, the subset of $\calS_G$ of points $\eta$ such that $\eta$ and $\xi$ are not transverse has proper Zariski closure in the quotient $\mathbf{G}/\mathbf{Q}$. We will exploit this fact in the proof of Proposition \ref{prop_transverse}.

Let $g_{\calX}$ be the Riemannian tensor on $\calX$ and let $\mathcal{J}_{\calX}$ be the $g_{\calX}$-invariant complex structure. The \emph{K\"{a}hler form} on $\calX$ is the differential 2-form $\omega_{\calD_{\calX}} \in \Omega^2(\calX)^G$ defined by
$$
(\omega_{\calD_{\calX}})_a(X,Y):=(g_{\calX})_a(X,(\mathcal{J}_{\calX})_a(Y)) \ ,
$$
where $a \in \calX$ and $X,Y \in T_a\calX$. Notice that being $G$-invariant, $\omega_{\calD_{\calX}}$ is closed by Cartan's Lemma \cite[VII.4]{Hel01}. For any triple of points $x,y,z \in \calD_{\calX}$, we can define the function 
$$
\beta_{\textup{Berg}}(x,y,z):=\int\limits_{\Delta(x,y,z)} \omega_{\calD_{\calX}}
$$
where $\Delta(x,y,z)$ denotes a smooth oriented triangle with geodesic edges and vertices $x,y,z$. 

We want to relate the function $\beta_{\textup{Berg}}$ with the notion of Bergman kernels. We recall that the space of complex-valued square integrable holomorphic functions $\mathcal{H}^2(\calD_{\calX})$ on $\calD_{\calX}$ is an Hilbert space with Hermitian product
$$
(f|g):=\int_{\calD_{\calX}} f(z)\overline{g(z)}d\mathcal{L}(z) \ ,
$$
where $\mathcal{L}$ is the Lebesgue measure on $\calD_{\calX}$. Since the evaluation of any element $f \in \mathcal{H}^2(\calD_{\calX})$ on a point $w \in \calD_{\calX}$ is a bounded linear functional on $\mathcal{H}^2(\calD_{\calX})$ (see  \cite{Kor00}), by the Riesz representation theorem it can be written as the Hermitian product $f(w)=(f|K_w)$, for some $K_w \in \mathcal{H}^2(\calD_{\calX})$. We can define the \emph{Bergman kernel} $k_{\calD_{\calX}}:\calD_{\calX} \times \calD_{\calX} \rightarrow \mathbb{C}^\ast$ as the Hermitian product $k_{\calD_{\calX}}(z,w)=(K_z|K_w)$. The following equation holds \cite[Theorem 3.7]{BIW07} 
\begin{equation}\label{eq_bergmann_relation}
\beta_{\textup{Berg}}(x,y,z):= -(\argu k_{\calD_{\calX}}(x,y)+\argu k_{\calD_{\calX}}(y,z)+\argu k_{\calD_{\calX}}(z,x)) \ ,
\end{equation}
where $\argu$ is the branch of the argument function with values between $-\pi$ and $\pi$. 

Since $k_{\calD_{\calX}}$ can be extended to pairs of transverse points in $\calS_G^{(2)}$ \cite{satake:1980}, by Equation \eqref{eq_bergmann_relation} the function $\beta_{\textup{Berg}}$ can be extended to the subset 
$$\calS_G^{(3)}\coloneqq \left\{(\eta_0,\eta_1,\eta_2 \in \calS_G^3 \,|\, (\eta_i,\eta_j)\in \calS_G^{(2)} \text{ if } i\neq j ) \right\}\subset \overline{\calD_{\calX}}^{3}\,.$$
If we still denote by $\beta_{\textup{Berg}}$ such extension, by \cite[Corollary 3.8]{BIW07} $\beta_{\textup{Berg}}$ is a continuous $G$-invariant alternating function and its Alexander-Spanier coboundary vanishes, that is $d\beta_{\textup{Berg}}(x_0,x_1,x_2,x_3)=0$ for every $4$-tuple $(x_0,x_1,x_2,x_3)$ with $(x_i,x_j,x_k)\in \calD_{\calX}^{3}\sqcup \calS_G^{(3)}$, whenever $i< j < k$. 
Furthermore, it satisfies 
$$\sup\limits_{(\eta_0,\eta_1,\eta_2)\in \calS_G^{(3)}} |\beta_{\textup{Berg}}(\eta_0,\eta_1,\eta_2)| = \pi\rk \calX , $$
where $\rk \calX$ denotes the \emph{real rank} of $\calX$, that is the dimension of a maximal flat embedded isometrically in $\calX$. For instance, when $G=\textup{SU}(p,q)$ the real rank of the associated symmetric space is given by $\min \{ p, q \}$. 
By \cite[Theorem 4.2]{BIW07} the restriction $\beta_{\textup{Berg}}:\calS_G^{(3)}\rightarrow \matR$ extends to a $G$-invariant Borel cocycle on the whole $\calS_G^3$. We call such extension the \emph{Bergman cocycle} and we denote it by $\beta_G$.

We conclude the section by introducing the definition of the Hermitian triple product and by relating it to the function $\beta_{\textup{Berg}}$. Consider the function
$$\langle\cdot,\cdot,\cdot\rangle:\calS_{G}^{(3)}\rightarrow\matC^*\; ,$$
$$\langle z_1,z_2,z_3\rangle\coloneqq k_{\calD_{\calX}}(\eta_0,\eta_1)k_{\calD_{\calX}}(\eta_1,\eta_2)k_{\calD_{\calX}}(\eta_2,\eta_0) \ ,
$$
where we have tacitly exploited the extension of the Bergman kernel to the boundary.
By \cite[Proposition 2.12]{BIW07} this map is continuous and Equation \eqref{eq_bergmann_relation} implies that
\begin{equation}\label{eq_relation}
\langle z_1,z_2,z_3\rangle\equiv e^{i\beta_{\textup{Berg}}(\eta_0,\eta_1,\eta_2)}\;\;\; \textup{mod}\;\; \matR^*
\end{equation}
for any $(\eta_0,\eta_1,\eta_2)\in \calS_G^{(3)}$. The notation $\textup{mod} \  \mathbb{R}^\ast$ refers to the fact that the two terms in Equation \eqref{eq_relation} differ by the multiplication of a non-zero real number. If we quotient $\matC^*$ by the $\matR^*$-action by dilations, we can compose the map $\langle \cdot , \cdot , \cdot \rangle$ with the quotient projection to obtain a map 
$$\langle\langle\cdot,\cdot,\cdot\rangle\rangle:\calS_G^{(3)}\rightarrow\matR^*\backslash\matC^* \  .$$
The map above is the \emph{Hermitian triple product}. 

In virtue of the identification between $\calS_G$ and the quotient $(\textbf{G}/\textbf{Q})(\matR)$, Burger, Iozzi and Wienhard extended the Hermitian triple product to a \emph{complex Hermitian triple product} defined on the whole $\mathbf{G}/\mathbf{Q}$. We start denoting by $A^*$ the group $\matC^*\times\matC^*$ endowed with the antilinear involution $(\lambda,\mu)\mapsto (\overline{\mu},\overline{\lambda})$ (that is a \emph{real structure} on $A^\ast$) and by $\Delta^*$ the image of the diagonal embedding of $\matC^*$.

By \cite[Corollary 2.17]{BIW07} there exists a rational map 
$$\langle\langle\cdot,\cdot,\cdot\rangle\rangle_{\matC}:(\textbf{G}/\textbf{Q})^{3}\rightarrow\Delta^*\backslash A^*$$ 
that makes the following diagram commutative
$$
\xymatrix{
\calS_G^{(3)} \ar[rr]^{\langle\langle\cdot,\cdot,\cdot\rangle\rangle}\ar[d]^{(i)^3} & & \matR^*\backslash\matC^* \ar[d]^{\Delta} \\
(\textbf{G}/\textbf{Q})^{3} \ar[rr]^{\langle\langle\cdot,\cdot,\cdot\rangle\rangle_{\matC}}& &\Delta^*\backslash A^*.
}
$$
Here $i:\calS_G\rightarrow \textbf{G}/\textbf{Q}$ refers to the $G$-equivariant identification between $\calS_G$ and $(\textbf{G}/\textbf{Q})(\matR)$, and $\Delta$ stands for the map induced by the diagonal embedding. It is worth mentioning that the complex Hermitian triple product is a rational function on $(\mathbf{G}/\mathbf{Q})^3$ since it can be written as a product of determinants of complex automorphy kernels (see \cite[Section 2.4]{BIW07} for more details).

The fact that a Hermitian space is not of tube type has consequences on the complex Hermitian triple product, and this fact is crucial in the proof of Theorem \ref{main_theorem}. For any pair of transverse points $(\eta_0,\eta_1)\in \calS_G^{(2)}$ we denote by $\calO_{\eta_0,\eta_1}\subset (\textbf{G}/\textbf{Q})(\matR)$ the Zariski open subset such that the map 
$$P_{\eta_0,\eta_1}:\calO_{\eta_0,\eta_1}\rightarrow \Delta^*\backslash A^*\;, \;\;\; \eta\mapsto \langle\langle \eta_0,\eta_1,\eta\rangle\rangle_{\matC}$$
is defined.  We have the following
\begin{lemma}[{\cite[Lemma 5.1]{BIW07}}]\label{lemma_constant}
Fix any $m\in\matZ$. Then the map
$$\calO_{\eta_0,\eta_1}\rightarrow\Delta^*\backslash A^*\;, \;\;\; \eta\mapsto P_{\eta_0,\eta_1}(\eta)^m$$ 
is not constant if and only if $\calX$ is not of tube type.
\end{lemma}

\begin{oss}
The \emph{if} part of Lemma \ref{lemma_constant} is exactly \cite[Lemma 5.1]{BIW07}, while the converse implication is a consequence of the characterization of Hermitian spaces not of tube type given in \cite[Theorem 1]{BIW07}. This part in particular provides an obstruction to the extension of Theorem \ref{main_theorem2}, since the arguments that we are going to use in the proof cannot be adapted in the tube type case.
\end{oss}

\subsection{Measurable cocycles}\label{sec_measurable_cocycles}
We now introduce the basic notions about the theory of measurable cocycles.
Let $\Gamma$ be a finitely generated discrete group and let $G$ be a locally compact group, both endowed with their Haar measurable structures. We denote by
$(X,\mu_X)$ a \emph{standard Borel probability $\Gamma$-space}, namely a probability space which is Borel isomorphic to a Polish space and endowed with a probability measure preserving $\Gamma$-action. 

\begin{defn}\label{cocycle_definition}
 A \emph{measurable cocycle} is a Borel measurable function $\sigma: \Gamma\times X\rightarrow G$
 which satisfies the following condition 
 \begin{equation}\label{cochain_condition}
 \sigma(\gamma_1\gamma_2,x)=\sigma(\gamma_1,\gamma_2x)\sigma(\gamma_2,x)
 \end{equation}
for every $\gamma_1,\gamma_2\in \Gamma$ and for almost every $x\in X$. 
\end{defn}

We can introduce an equivalence relation on the set of measurable cocycles. 

\begin{defn}\label{def_cohomologous}
 Let $\sigma_1,\sigma_2:\Gamma\times X \rightarrow G$ be two measurable cocycles, let 
 $f:X\rightarrow G$ be a measurable map and denote by $\sigma_1^f$ the cocycle defined by
\begin{equation}\label{eq_cohomology}
\sigma_1^f(\gamma,x)\coloneqq f(\gamma x)^{-1}\sigma_1(\gamma,x)f(x) \ ,
\end{equation}
 for every $\gamma\in \Gamma$ and almost every $x\in X$. The cocycle $\sigma^f_1$ is the \emph{$f$-twisted cocycle associated to $\sigma_1$}. 
 We say that $\sigma_1$ is \emph{cohomologous} to $\sigma_2$ (writing $\sigma_1\simeq\sigma_2$) if there exists a measurable map $f$
 such that $\sigma_2=\sigma_1^f$.
\end{defn}

The words cocycle and cohomologous refer to the \emph{cohomology theory of countable equivalence relations} introduced by Feldman and Moore \cite{moore1976,feldman:moore}. Even if their theory computes the cohomology of a generic countable equivalence relation with values in an abelian Polish group, we can adapt it to compute the $1$-dimensional cohomology of a specific class of countable equivalence relations, namely \emph{orbital equivalence relations}, with values into any locally compact group. Precisely, in the setting of Definition \ref{cocycle_definition}, one can define the equivalence relation $\calR_{\Gamma}\subset X \times X$, where $(x,y)\in \calR_{\Gamma}$ if and only if $y=\gamma x$ for some $ \gamma \in \Gamma$. Given a locally compact group $G$, we define the space $Z^1(\calR_{\Gamma};G)$ as the set of measurable functions $c:\calR_{\Gamma }\rightarrow G$ satisfying the relation 
 \begin{equation}\label{equation_relation_condition}
  c(x,z) =c(y,z)c( x, y)
\end{equation} 
for almost every $(x,y),(x,z),(y,z) \in \calR_\Gamma$. In this way, we get a natural identification between measurable cocycles and the set $Z^1(\calR_{\Gamma};G)$ realized by the following map
$$
\Theta: \{ \sigma:\Gamma \times X \rightarrow G \ | \ \sigma \ \textup{is a cocycle} \} \rightarrow Z^1(\calR_\Gamma;G) \ ,
$$
$$
\sigma \mapsto c_\sigma(x,\gamma x):=\sigma(\gamma,x) \ .
$$

The \emph{$1$-cohomology group of $\calR_\Gamma$ with values in $G$}, denoted by $\Hm^1(\Gamma \curvearrowright X;G)$, is defined as the quotient $Z^1(\calR_{\Gamma};G)/\sim$, where 
two functions $c_1, c_2:\calR_\Gamma \rightarrow G$ are equivalent if there exists a measurable function $f:X \rightarrow G$ such that $$c_2(x,\gamma x)=f(\gamma x)^{-1} c_1(x,\gamma x) f(x) \ ,$$ for every $\gamma \in \Gamma$ and almost every $x\in X$. It is worth noticing that the condition $c_2(x,\gamma x)=f(\gamma x)^{-1} c_1(x,\gamma x) c(x)$ is exactly the one of Equation \eqref{eq_cohomology} applied to the cocycles $\Theta^{-1}(c_1)$ and $\Theta^{-1}(c_2)$. In other words, the map $\Theta$ factors through the equivalence relation of cohomology between cocycles and defines a bijection 
\begin{align*}
\{\sigma:\Gamma  \times X\rightarrow G\;,\; \sigma\text{ cocycle } \}/_{\simeq}& \leftrightarrow\Hm^1(\Gamma \curvearrowright X;G)\,.
\end{align*}

Even when two measurable cocycles have different targets, it is still possible to introduce an equivalence relation between them.

\begin{defn}\label{def_equivalent_cocycles}
Given two measurable cocycles $\sigma_1:\Gamma \times X \rightarrow G_1$ and $\sigma_2:\Gamma \times X \rightarrow G_2$ with different targets, we say that $\sigma_1$ and $\sigma_2$ are \emph{equivalent} (writing $\sigma_1 \sim \sigma_2$) if there exists a group isomorphism $s:G_1 \rightarrow G_2$ such that $s \circ \sigma_1 \simeq \sigma_2$. 
\end{defn}

Any representation $\rho:\Gamma\rightarrow G$ determines naturally a measurable cocycle $\sigma_{\rho}:\Gamma\times X\rightarrow G$ by setting $\sigma_{\rho}(\gamma,x)\coloneqq \rho(\gamma)$. If two representations are $G$-conjugated, then the associated cocycles are cohomologous via a constant measurable function. In this way we can view the space $\Hm^1(\Gamma \curvearrowright X; G)$ as the natural generalization of the character variety $\textup{Rep}(\Gamma;G)$, namely the space of representations $\Gamma \rightarrow G$ modulo $G$-conjugation.  

Another important tool in the study of representations into a semisimple algebraic group is given by the Zariski closure of the image. In this context, since the image is a subgroup of the target group, its Zariski closure is a group as well. In order to get an analogous definition for measurable cocycles, we need to introduce the notion of algebraic hull. Such notion is necessary since a priori the image of a measurable cocycle is not a subgroup.

\begin{defn}\label{def:algebraic:hull}
Let \textbf{G} be a semisimple real algebraic group and define $G=\mathbf{G}(\mathbb{R})^\circ$. The \emph{algebraic hull} of a measurable cocycle $\sigma: \Gamma \times X\rightarrow G$ is the $G$-conjugacy class of the smallest algebraic subgroup $\textbf{L}$ of $\textbf{G}$ such that $\textbf{L}(\matR)^{\circ}$ contains the image of a
cocycle cohomologous to $\sigma$.

We say that $\sigma$ is \emph{Zariski dense} if it holds $\mathbf{G}=\mathbf{L}$. 
\end{defn}

We need to introduce the conjugacy class of the subgroup to define the algebraic hull, since the cohomology relation allows us to twist the cocycle using $G$-conjugacy. We refer to \cite[Proposition 9.2.1]{zimmer:libro} for a proof of the fact that Definition \ref{def:algebraic:hull} is well-defined thanks to the Noetherianity of the target. In this paper we will be interested in the subset of Zariski dense cocycles in $\Hm^1(\Gamma \curvearrowright X; G)$, that we denote by $\Hm_{ZD}^1(\Gamma \curvearrowright X; G)$. The latter is merely a subset and it has no other algebraic structure. 

\subsection{Boundary maps}\label{section_boundary_maps}
Another important tool in the theory of measurable cocycles is the notion of boundary map. In this section we are going to recall an existence result for ergodic Zariski dense cocycles \cite[Theorem 1]{sarti:savini1}. We will focus our attention on the main properties of the \emph{slices} of such boundary maps. 

We start introducing the notion of boundary for a discrete countable group. We follow Bader and Furman \cite{BF14}. We first recall the notion of amenable action. Given a locally compact second countable group $H$, a \emph{Lebesgue $H$-space} $(S,\nu)$ is a standard Borel probability space where the $H$-action preserves only the measure class of $\nu$. A \emph{mean} on $\Linf (H\times S;\matR)$ is a norm-one linear operator 
$$m: \Linf (H\times S;\matR)\rightarrow \Linf (S;\matR) \ , $$ 
such that $m(\chi_{H \times S})=\chi_S$, $m(f) \geq 0$ whenever $f$ is a positive function and for all $f \in \Linf(H \times S)$ and any measurable subset $A \subset S$ it holds $m(f \cdot \chi_{H \times A})=m(f)\cdot \chi_A$. An action of $H$ on $S$ is \emph{amenable}, or equivalently $S$ is an \emph{amenable $H$-space} \cite[Section 5.3]{monod:libro}, if there exists a $H$-equivariant mean on $\Linf (H \times S;\matR)$. An example of amenable action is given by the action of a lattice $\Gamma < H$ in a semisimple Lie group on a homogeneous $H$-space of the form $H/L$, where $L$ is an amenable subgroup of $H$ (for instance a minimal parabolic subgroup). 

We now restrict to a discrete countable group $\Gamma$. Given an equivariant measurable map $p:U \rightarrow V$ between two Lebesgue $\Gamma$-spaces, a \emph{metric along $p$} is a Borel function $d: U \times_p U \rightarrow [0,\infty)$ on the fibered product whose restriction $d_v$ to the fiber $U_v:=p^{-1}(v)$ determines a separable metric space. We say that the $\Gamma$-action is \emph{fiberwise isometric} if any $\gamma \in \Gamma$ acts isometrically on the fibers of $p$, that is $\gamma:U_v \rightarrow U_{\gamma.v}$ is an isometry, namely 
$$
d_{\gamma.v}(\gamma.x,\gamma.y)=d_v(x,y) \ ,
$$
for every $\gamma \in \Gamma, v \in V, x,y \in U_v$. A measurable map $q:Y \rightarrow Z$ between Lebesgue $\Gamma$-spaces is \emph{relatively metrically ergodic} \cite[Definition 2.1]{BF14} if for any fiberwise isometric $\Gamma$-action along a map $p:U \rightarrow V$ and any measurable $\Gamma$-equivariant maps $f:Y \rightarrow U$ and $g:Z \rightarrow V$, there exists a $\Gamma$-equivariant measurable map $\psi:Z \rightarrow U$ such that the following diagram commutes
\begin{equation}\label{eq:lift:diagram}
\xymatrix{
Y \ar[rr]^f \ar[d]^q && U \ar[d]^p \\
Z \ar[rr]^g \ar@{.>}[urr]^\psi && V \ . 
}
\end{equation}

 \begin{defn}\label{def_boundary}
 Let $\Gamma$ be a discrete countable group. A $\Gamma$\emph{-boundary} is an amenable $\Gamma$-space $(B,\nu)$ such that the projections $\pi_1:B\times B\rightarrow B$ and $\pi_2:B\times B\rightarrow B$ on the first and the second factor, respectively, are relatively metrically ergodic.
 \end{defn}

\begin{es}\label{es_gamma_boundary}
We report two examples of $\Gamma$-boundary. The first one is a model valid for any discrete finitely generated group, whereas the second one is specific for lattices in Lie groups. 
 
\begin{enumerate}
\item Let $\Gamma$ be a discrete finitely generated group and let $S$ be a symmetric set of generators. Following the line of Burger and Iozzi \cite{BI04} we recall the construction of a model for a $\Gamma$-boundary. We define a probability measure on $\Gamma$ as
$$\mu_S=\frac{1}{2|S|}\sum\limits_{s\in S} \delta_s+\delta_{s^{-1}}.$$ 

We start constructing the realization of a boundary for the free group $\matF_S$ on the set $S$.  Let $\calT_S(\infty)$ be the boundary of the Cayley graph $\calT_S$ of $\matF_S$, namely the set of all reduced words on $S$ of infinite length. We endow such a boundary 
with the $\matF_S$-quasi-invariant measure defined by
$$\overline{m}(C(x))=\frac{1}{2r(2r-1)^{n-1}}$$
where $x$ is any reduced word of length $n$, $r=|S| $ and $C(x)$ denotes the set of all reduced words of infinite length starting with $x$. One has that $(\calT_S(\infty),\overline{m})$ is a $\matF_S$-boundary. 

Coming back to $\Gamma$, if $\rho:\matF_S\rightarrow \Gamma$ is the representation of $\Gamma$ realizing it as a quotient,  we denote by $N=\ker \rho$ and we consider the set $\Linf(\calT_S(\infty),\overline{m})^N$ of $N$-invariant essentially bounded functions on $\calT_S(\infty)$. By the Mackey realization theorem \cite{Mackey}, there exists a standard measure space $(B,\nu)$ equipped with a measurable map $p:\calT_S(\infty)\rightarrow B$ such that $p_*(\overline{m})=\nu$ and the pull back via $p$ identifies $\Linf(B,\nu)$
with $\Linf(\calT_S(\infty),\overline{m})^N$. By \cite[Theorem 2.7]{BF14} we have that $(B,\nu)$ is a $\Gamma$-boundary. 

\item When $\Gamma$ is a lattice in a semisimple Lie group $H$, the description of a $\Gamma$-boundary becomes easier since it can be identified with the $H$-homogeneous quotient $H/P$, where $P<H$ is any minimal parabolic subgroup \cite[Theorem 2.3]{BF14}. For instance, if $\Gamma$ is a lattice in a simple Lie group of real rank one, a $\Gamma$-boundary coincides with the visual boundary of the associated hyperbolic space. 
\end{enumerate}
\end{es}

Now we are ready to give the definition of boundary map, that we formulate in our specific case.
\begin{defn}\label{def_boundary_map}
Let $\Gamma$ be a finitely generated group and let $G$ be a locally compact group. Consider a standard Borel probability $\Gamma$-space $(X,\mu_X)$, a $\Gamma$-boundary $B$ and a Lebesgue $G$-space $(Y,\nu)$. A \emph{boundary map} for a measurable cocycle $\sigma:\Gamma\times X\rightarrow G$ is a measurable map 
$$\phi:B \times X \rightarrow Y 
$$
which is \emph{$\sigma$-equivariant}, that is
$$
\phi(\gamma b,\gamma x)=\sigma(\gamma,x)\phi(b,x) 
$$
for every $\gamma \in \Gamma$ and almost every $b \in B, x \in X$. 
\end{defn}

\begin{oss}
In Definition \ref{def_boundary_map} the measure $\mu_X$ is $\Gamma$-invariant, whereas the measure on the boundary $B$ is only \emph{quasi-invariant} (only its measure class is invariant). 
\end{oss}

\begin{oss}\label{oss_twisted_map}
Given a cocycle $\sigma:\Gamma\times X\rightarrow G$ and a measurable function $f:X\rightarrow G$, a boundary map $\phi:B\times X\rightarrow Y$ naturally defines a boundary map for the $f$-twisted cocycle $\sigma^f$ as 
$$\phi^f(b,x)\coloneqq f(x)^{-1} \phi(b,x).$$
\end{oss}

In this paper the space $Y$ appearing in Definition \ref{def_boundary_map} will be the Shilov boundary $\calS_G$ (Definition \ref{def_shilov_boundary}) of a simple Hermitian group $G$ not of tube type. Such space is neither a $G$-boundary nor a $\Gamma$-boundary for any lattice $\Gamma<G$), since the $G$-action on it is not amenable in general (unless the real rank of $G$ is one). In fact  $\calS_G$ can be identified with $G/Q$, where $Q$ corresponds to the real points of a maximal parabolic subgroup, and the latter is amenable only if $G$ has rank one. Nevertheless, given a minimal parabolic subgroup $P<G$, by the fact that $P<Q$, there always exists a $G$-equivariant map $G/P \rightarrow G/Q$. For instance, when $G=\textup{SU}(p,q)$, the $G$-boundary $G/P$ parametrizes maximal complete flags of isotropic subspaces in $\mathbb{C}^{p+q}$ with respect to the Hermitian form $h_{p,q}$, whereas the Shilov boundary $G/Q$ parametrizes only isotropic subspaces of maximal dimension. In that case the map $G/P \rightarrow G/Q$ sends the flag to the space of maximal dimension appearing in the flag itself. 

Given a measurable cocycle $\sigma:\Gamma \times X \rightarrow G$ and a boundary map $\phi:B \times X \rightarrow Y$, we can define 
$$
\phi_x:B \rightarrow Y \ , \ \ \ \phi_x(b):=\phi(b,x) \ ,
$$
for almost every $b \in B,x \in X$. By \cite[Chapter VII, Lemma 1.3]{margulis:libro} the map $\phi_x$ is measurable and it is called the $x$-\emph{slice} of the boundary map $\phi$. By the $\sigma$-equivariance of $\phi$ we have that 
$$
\phi_{\gamma x}( \ \cdot \ )=\sigma(\gamma,x)\phi_x( \ \cdot \ ) \ ,
$$
for every $\gamma \in \Gamma$ and almost every $x \in X$. When $G$ is a connected simple Lie group (recall that $G=\textbf{G}(\matR)^{\circ}$ where $\textbf{G}$ is the connected adjoint $\matR$-group associated to the complexification of the Lie algebra of $G$) and $Y$ coincides with the real points of a quasi projective variety of the form $\mathbf{G}/\mathbf{L}$ for some real algebraic subgroup $\mathbf{L}<\mathbf{G}$, we say that the $x$-slice is \emph{Zariski dense} if the Zariski closure of the essential image of $\phi_x$ is exactly $\mathbf{G}/\mathbf{L}$.  

In case of Zariski dense cocycles, if $\Gamma$ acts ergodically on $X$, we have the following result.

\begin{thm}\cite[Theorem 1, Proposition 4.4]{sarti:savini1}\label{thm_boundary}
Let $\Gamma$ be a finitely generated discrete group with a $\Gamma$-boundary $B$ and let $G$ be a simple Lie group of non-compact type. Let $(X,\mu_X)$ be an ergodic standard Borel probability $\Gamma$-space and consider a Zariski dense cocycle $\sigma:\Gamma\times X \rightarrow G$. Then $\sigma$ admits a boundary map $\phi: B\times X\rightarrow G/P$, where $P<G$ is a minimal parabolic subgroup. Additionally the slices of $\phi$ are Zariski dense. 
\end{thm}

When $G$ is a simple Lie group of Hermitian type, thanks to Theorem \ref{thm_boundary} we can compose the boundary map $B \times X \rightarrow G/P$ with the $G$-equivariant map $G/P \rightarrow G/Q$ to obtain a boundary map $\phi:B \times X \rightarrow G/Q$ in the Shilov boundary of $G$. We are going to prove that the slices of such map preserve transversality. 

\begin{prop}\label{prop_transverse}
 Let $\sigma:\Gamma\times X\rightarrow G$ be a Zariski dense measurable cocycle with a boundary map $\phi : B\times X\rightarrow G/Q$. Then for almost every $x\in X$ and $b_1,b_2\in B$,
 $\phi(b_1,x),\phi(b_2,x)$ are transverse in $G/Q$.
\end{prop}
\begin{proof}
We consider the set 
$$
E:=\{ (b_1,b_2,x) \in B \times B \times X \ | \ \phi(b_1,x) \ \textup{is not transverse to}\ \phi(b_2,x) \ \} \ .
$$
The previous set is measurable, by the measurability of the map $\phi$, and it is also $\Gamma$-invariant, by the $\sigma$-equivariance of $\phi$. By the ergodicity of the $\Gamma$-action on the product $B\times B\times X$ \cite[Proposition 2.4]{MonShal0}, the set $E$ has either full or zero measure. We claim that $E$ must have measure zero.

By contradiction, suppose that $E$ has full measure. By Fubini's theorem, there must exist a point $b_0 \in B$, such that 
\begin{equation}\label{eq:non:transverse}
\phi_x(b) \ \textup{is not transverse to} \ \phi_x(b_0) \ ,
\end{equation}
for almost every $x \in X,b \in B$, where $\phi_x$ is the $x$-slice of $\phi$. If we denote by $\textbf{nt}(\phi_x(b_0))$ the subset of $G/Q$ of points non-transverse to $\phi_x(b_0)$, Condition \eqref{eq:non:transverse} implies that essential image $V_x \coloneqq \textup{EssIm}(\phi_x)$ lies inside $\textbf{nt}(\phi_x(b_0))$ for almost every $x \in X$. 

On one hand, by Section \ref{sec:hermitian:spaces}, the set of all points that are not transverse to $\phi_x(b_0)$ has proper Zariski closure, hence the Zariski closure $V_x$ lies in a proper Zariski closed subset of $\mathbf{G}/\mathbf{Q}$. On the other hand, Theorem \ref{thm_boundary} implies that each slice is Zariski dense, thus the Zariski closure of $V_x$ cannot lie in any proper Zariski closed subset of $\mathbf{G}/\mathbf{Q}$. This leads to a contradiction and proves the $E$ has measure zero. Equivalently 
$$
\phi(b_1,x) \ \textup{is transverse to}\ \phi(b_2,x) \ 
$$
for almost every $x \in X$ and $b_1,b_2 \in B$, as claimed. 
\end{proof}

We can sum up all we have shown so far in the following 

\begin{cor}\label{cor_boundary}
Let $\Gamma$ be a finitely generated group with $\Gamma$-boundary $B$ and let $G$ be a Hermitian Lie group. Let $(X,\mu_X)$ be an ergodic standard Borel probability $\Gamma$-space and consider a Zariski dense cocycle $\sigma:\Gamma\times X \rightarrow G$. Then there exists a boundary map $\phi: B\times X\rightarrow G/Q$, where $Q<G$ is a maximal parabolic subgroup. Moreover, the slices of $\phi$ are Zariski dense and for almost every $x\in X$ and $b_1,b_2\in B$, it holds $(\phi(b_1,x),\phi(b_2,x))\in (G/Q)^{(2)}$.
\end{cor}

\subsection{Bounded cohomology}\label{section_bounded_cohomology}
In this section we briefly recall the theory of continuous and continuous bounded cohomology. We refer the reader both to \cite{burger2:articolo,monod:libro} for more details about the functorial approach and to \cite{moraschini:savini,moraschini:savini:2} for a more detailed discussion about the pullback induced by measurable cocycles.

Let $G$ be a locally compact group and let $E$ be a Banach $G$-module, that is a Banach module equipped with an isometric action $\pi:G\rightarrow\Isom(E)$. We assume that $E$ is the dual of some Banach space and we endow it with both the weak$^*$ topology and the associated Borel structure.

The space of $E$-\emph{valued continuous functions} on $G$ is
$$\textup{C}^{\bullet}_c(G;E)\coloneqq \left\{f:G^{\bullet+1}\rightarrow E\;|\; f \text{ continuous } \right\} \ , $$
and admits as a subspace the set of continuous \emph{bounded} functions on $G$, namely  
$$\textup{C}^{\bullet}_{\textup{cb}}(G;E)\coloneqq \left\{f\in\textup{C}^{\bullet}_c(G;E)\;|\; ||f||_{\infty}<+\infty\right\} \ ,$$
where 
$$||f||_{\infty}=\sup\limits_{g_0,\ldots,g_{\bullet} \in G} ||f(g_0,\ldots,g_{\bullet})||_E \ . $$
 The \emph{standard homogeneous coboundary operator} $$\delta^{\bullet}:\textup{C}^{\bullet}_c(G;E)\rightarrow \textup{C}^{\bullet+1}_c(G;E) \ , $$ 
$$\delta^{\bullet}f(g_0,\ldots,g_{\bullet+1})\coloneqq \sum\limits_{i=0}^{\bullet+1} (-1)^i f(g_0,\ldots,g_{i-1},g_{i+1},\ldots,g_{\bullet+1}) \ ,$$
preserves both continuity and boundedness and it allows to the define a cochain complex $(\textup{C}^\bullet_{\textup{c(b)}}(G;E),\delta^\bullet)$. 

One can consider the subspace of $G$-\emph{invariant} (bounded) $E$-valued continuous functions on $G$ as the set
$$\textup{C}^{\bullet}_{\textup{c(b)}}(G;E)^G\coloneqq \{f\in \textup{C}_{\textup{c(b)}}^{\bullet}(G;E)\;| \; gf=f \;,\; \forall g\in G \} \ ,$$
where the $G$-action is given by $(g f)(g_0,\ldots,g_{\bullet})\coloneqq \pi(g)f(g^{-1}g_0,\ldots,g^{-1}g_{\bullet})$ for every $g,g_0,\ldots,g_{\bullet}\in G$. Since $\delta^{\bullet}$ also preserves $G$-invariant cochains, we can give the following 
\begin{defn}
The \emph{continuous (bounded) cohomology} of $G$ with coefficients in the Banach $G$-module $E$ is the cohomology of the complex $(\textup{C}^{\bullet}_c(G;E)^{G},\delta^{\bullet})$ (respectively $(\textup{C}^{\bullet}_{\textup{cb}}(G;E)^{G},\delta^{\bullet})$) and it is denoted by $\Hc^{\bullet}(G;E)$ (respectively $\Hcb^{\bullet}(G;E)$).
\end{defn}

\begin{oss}
When we deal with a discrete group $\Gamma$, the continuity condition is trivially satisfied. To lighten the notation, we will write $\Hm^{\bullet}(\Gamma;E)$ for the continuous cohomology and
$\Hb^{\bullet}(\Gamma;E)$ for the continuous bounded cohomology of $\Gamma$.
\end{oss}

%

Given a $G$-equivariant map $\alpha:E\rightarrow F$ between two Banach $G$-modules $E$ and $F$, we can consider the \emph{change of coefficients} a the level of continuous (bounded) cohomology groups
\begin{equation}
\Hcbb^{k}(\alpha):\Hcbb^k(G;E)\rightarrow \Hcbb^k(G;F)
\end{equation}
for every $k\geq 0$. In this paper we are going to deal with two Banach modules:
\begin{itemize}
\item $\matR$, endowed with the trivial $G$-action;
\item $\Linf(X;\matR)$, the space of essentially bounded functions on a standard Borel probability space $G$-space $(X,\mu_X)$. It will be endowed with the $G$-action $(g f)(x)\coloneqq f(g^{-1}x)$, for every $g \in G$ and $f \in \textup{L}^\infty(X;\mathbb{R})$.
\end{itemize}
The inclusion $\matR\hookrightarrow \Linf(X;\matR)$ induces maps
\begin{equation}\label{eq_inclusion}
\Hcbb^k(G;\matR)\rightarrow \Hcbb^k(G;\Linf(X;\matR))
\end{equation}
for every $k\geq0$.

The notion of bounded cohomology turns out to be as simple to define as hard to apply for computations. 
An extremely powerful tool was provided by Burger and Monod \cite{monod:libro,burger2:articolo}. They showed a way to compute bounded cohomology of locally compact groups using strong resolutions by relatively injective modules. Since the theory is quite technical, we omit it and we refer to Monod's book \cite{monod:libro} for a more detailed discussion. We only recall the strong resolution of the essentially bounded  weak$^*$ measurable functions on the boundary of a group. More precisely, let $\Gamma$ be a discrete countable group and let $(B,\nu)$ a $\Gamma$-boundary. Let $\Linf_{\textup{w}^\ast}(B^{\bullet+1};E)$ be the space of (classes of) essentially bounded weak$^\ast$ measurable functions on $B^{\bullet+1}$ with values in a $\Gamma$-module $E$. With an abuse of notation we will use representatives to refer to elements of $\Linf_{\textup{w}^\ast}(B^{\bullet+1};E)$. We consider the complex $(\Linf_{\textup{w}^*}(B^{\bullet+1};E),\delta^{\bullet})$, where $\delta^{\bullet}$ is the standard homogeneous coboundary operator and the $\Gamma$-action is given by
$$(\gamma f)(b_0,\ldots,b_{\bullet})\coloneqq \pi(\gamma)f(\gamma^{-1}b_0,\ldots,\gamma^{-1}b_{\bullet})$$ for any $\gamma \in \Gamma$, $f\in  \Linf_{\textup{w}^*}(B^{\bullet+1};E)$ and $(b_0,\ldots,b_{\bullet})\in B^{\bullet+1}$.
By adding to the above complex the inclusion of coefficients $E\hookrightarrow \Linf_{\textup{w}^*}(S;E)$ and by taking the subresolution of $\Gamma$-invariants, we get an isometric isomorphism \cite[Corollary 1.5.3]{burger2:articolo}
\begin{equation}\label{eq_iso_cohomology}
\Hb^k(\Gamma;E)\cong \Hm^k(\Linf_{\textup{w}^*} (B^{\bullet+1};E)^{\Gamma},\delta^{\bullet})
\end{equation}
for every $k \geq 0$. 

\begin{es}\label{es_resolution_boundary}
When $\Gamma$ is a lattice in a semisimple Lie group $H$, by Example \ref{es_gamma_boundary} we know that a $\Gamma$-boundary is given by the quotient $H/P$, where $P$ is any minimal parabolic subgroup of $H$. As a consequence we can compute its bounded cohomology exploiting the resolution of essentially bounded weak$^\ast$ measurable functions on $H/P$. In this context even more is true: the $H$-invariants of the same complex computes also the continuous bounded cohomology of $H$ itself. 
\end{es}

In a similar way, one can exhibit isometric isomorphisms
\begin{equation}
\Hb^k(\Gamma;E)\cong \Hm^k(\Linf_{\textup{w}^*,\textup{alt}} (B^{\bullet+1};E)^{\Gamma},\delta^{\bullet})
\end{equation}
in any degree $k\geq 0$, where $(\Linf_{\textup{w}^*,\textup{alt}}(B^{\bullet+1};E),\delta^{\bullet})$ is the resolution of essentially bounded weak$^*$ measurable alternating functions on $B$ \cite[Corollary 1.5.3]{burger2:articolo}. In our context a function $f:B^{\bullet+1}\rightarrow E$ is \emph{alternating} if 
$$f(b_{\sigma(0)},\ldots,b_{\sigma({\bullet})})=\textup{sgn}(\sigma) f(b_0,\ldots,b_{\bullet})$$
for any permutation $\sigma\in\mathfrak{S}_{\bullet+1}$ and for every $(b_0,\ldots,b_{\bullet})\in B^{\bullet+1}$.

Using boundaries to compute pullback maps in bounded cohomology may reveal difficult. In general a boundary map may not preserve the measure classes on the boundaries and hence it may not define any map between the resolutions of essential bounded weak$^\ast$ measurable functions. For this reason Burger and Iozzi \cite{burger:articolo} suggested to exploit a different complex. Let $G$ be a locally compact group and let $Y$ be a Lebesgue $G$-space. We consider the complex $(\calB^{\infty}_{\textup{w}^*}(Y^{\bullet+1};E),\delta^{\bullet})$ of  weak$^\ast$ measurable bounded functions on $Y$ with the standard homogeneous coboundary operator. Notice that this time we are dealing with functions and not with equivalence classes of functions as in Example \ref{es_resolution_boundary}. Burger and Iozzi \cite[Corollary 2.2]{burger:articolo} proved that there exists a canonical non-trivial map 
\begin{equation}\label{equation_canonical_map}
  \mathfrak{c}^k :\Hm^k(\calB^{\infty}_{\textup{w}^*}(Y^{\bullet+1};E)^{G})\rightarrow \Hcb^k(G;E)
\end{equation} for every $k\geq 0$. The same holds for the alternating subcomplex. 

\begin{es}\label{es_implement_kahler_class}
Let $G$ be a semisimple Lie group of Hermitian type and let $\calS_G$ be its Shilov boundary. Since $\calS_G$ is identified with the quotient $G/Q$, where $Q$ is the subgroup obtained by intersecting $G$ with a maximal parabolic subgroup in the complexification, the space $\calS_G$ is a Lebesgue $G$-space. By Section \ref{sec:hermitian:spaces} we know that the Bergman cocycle $\beta_G$ is an everywhere defined $G$-invariant alternating cocycle, thus it can be viewed as an element
$$
\beta_G \in \calB^{\infty}_\textup{alt}(S^{3}_G;\mathbb{R})^G \ .
$$
One can verify \cite[Proposition 4.3]{BIW07} that the image of $[\beta_G]$ under the canonical map 
$$\mathfrak{c}^2:\Hm^2(\calB^{\infty}_\textup{alt}(\calS_G^{\bullet+1};\mathbb{R})^G)\rightarrow\Hcb^2(G;\matR)$$ is non-trivial.
Such a class is called \emph{bounded K\"{a}hler class} of $G$ and it is denoted by $k_G^b$.
\end{es}

We are now ready to recall the notion of pullback along measurable cocycles. Given a finitely generated group $\Gamma$ and a measurable cocycle $\gencoc$, we define the map 
$$\textup{C}_{b} ^{\bullet}(\sigma):\textup{C}^{\bullet}_{\textup{cb}}(G;\matR)^G\rightarrow \textup{C}_{b}^{\bullet}(\Gamma;\Linf (X;\matR))^{\Gamma}$$
as follows
$$\textup{C}_{b} ^{\bullet}(\sigma)(\psi)(\gamma_0,\ldots,\gamma_{\bullet})(x)\coloneqq\psi(\sigma(\gamma_0^{-1},x)^{-1},\ldots,\sigma(\gamma_{\bullet}^{-1},x)^{-1}).$$
Such a map actually is a cochain map \cite[Lemma 2.7]{savini2020}. Hence it descends to a map at the level of cohomology groups
$$\Hb ^k(\sigma):\Hcb^k(G;\matR)\rightarrow \Hb^k(\Gamma;\Linf(X;\matR))$$
for every $k\geq 0$. 

The map induced in bounded cohomology depends only on the cohomology class of $\sigma$. 

\begin{prop}\label{prop_pull_back_invariance}
Let $\Gamma$ be a finitely generated group and let $(X,\mu_X)$ be a standard Borel probability $\Gamma$-space. Given a measurable cocycle $\sigma:\Gamma \times X \rightarrow G$ and a measurable map $f:X \rightarrow G$, it holds that
$$
\Hb^{\bullet}(\sigma^f)=\Hb^\bullet(\sigma) \ .
$$
\end{prop}

Except for the different coefficients modules involved, the proof of Proposition \ref{prop_pull_back_invariance} is analogous to the one of \cite[Lemma 2.9]{savini2020} and for this reason we refer the reader there for more details. 

When $G$ is a group of Hermitian type, the pullback construction allows us to give the main definition of the paper.

\begin{defn}\label{def_parametrized_kahler_class}
Let $\Gamma$ be a finitely generated group and let $G$ be a Hermitian Lie group. Let $(X,\mu_X)$ be a standard Borel probability $\Gamma$-space and $\gencoc{}$ a measurable cocycle. The \emph{parametrized K\"{a}hler class} associated to $\sigma$ is the class $\Hb^2(\sigma)(k_G^b)\in \Hb^2(\Gamma;\Linf(X;\matR))$.
\end{defn}

We conclude the section showing how we can implement the parametrized K\"{a}hler class if the cocycle admits a boundary map. Let $\sigma:\Gamma \times X \rightarrow G$ be a measurable cocycle. If $\phi:B\times X\rightarrow Y$ is a boundary map for $\sigma$, we can naturally define a map at the level of cochains as 
$$\textup{C}^{\bullet}(\Phi):\calB^{\infty}(Y^{\bullet+1};\matR)^G\rightarrow \Linf_{\textup{w}^*}(B^{\bullet+1};\Linf(X;\matR))^{\Gamma},$$
$$\textup{C}^{\bullet}(\Phi)(\psi)(b_0,\ldots,b_{\bullet})(x)\coloneqq \psi(\phi(b_0,x),\ldots,\phi(b_{\bullet},x))$$
for every $\psi \in\calB^{\infty}(Y^{\bullet+1};\matR)^G $ and almost every $(b_0,\ldots,b_{\bullet})\in B^{\bullet+1}$ and $x\in X$.
The above map is a well-defined cochain map and it does not increase the norm \cite[Lemma 4.2]{moraschini:savini}. As a consequence, it induces maps at the level of cohomology groups
$$\Hm^k(\Phi):\Hm^k( \calB^{\infty}(Y^{\bullet+1};\matR)^G)\rightarrow \Hb^k(\Gamma;\Linf (X;\matR))$$
for every $k\geq 0$. 

An immediate application of \cite[Proposition 1.2]{burger:articolo} shows the commutativity of the following diagram
\begin{equation}\label{eq_general_diagram}
\xymatrix{
\Hm^k(\calB^{\infty}(Y^{\bullet+1};\matR)^G)\ar[rr]^{\hspace{20pt} \mathfrak{c}^k} \ar[d]^{\Hm^k(\Phi)}&&\Hcb^k(G;\matR)\ar[dll]^{\Hb^k(\sigma)}\\
\Hb^k(\Gamma;\Linf(X;\matR)) & 
}
\end{equation}
for every $k\geq0$. 

\begin{es}\label{es_pullback_kahler_boundary}
Let $\Gamma$ be a discrete countable group and let $(X,\mu_X)$ be an ergodic standard Borel probability $\Gamma$-space. Consider a Zariski dense measurable cocycle $\sigma:\Gamma \times X \rightarrow G$, where $G$ is a semisimple Lie group of Hermitian type. By Corollary \ref{cor_boundary} there exists a boundary map $\phi:B \times X \rightarrow \calS_G$ in the Shilov boundary of $G$. Thanks to Example \ref{es_implement_kahler_class} we know that the class $\frac{1}{2\pi}[\beta_G] \in \Hm^2(\calB^\infty(S_G^3;\mathbb{R})^G)$ is sent to the bounded  K\"{a}hler class $k^b_G$. Diagram \ref{eq_general_diagram} tells us that the class $\Hb^2(\sigma)(k^b_G)$ admits as a natural representative $\frac{1}{2\pi}\textup{C}^2(\Phi)(\beta_G)$, that is 
$$
\textup{C}^2(\Phi)(b_0,b_1,b_2)(x)=\beta_G(\phi(b_0,x),\phi(b_1,x),\phi(b_2,x)) \ .
$$
\end{es}

\begin{oss}\label{remark_holom}
Given a Hermitian symmetric space $\calX$ with $G=\Isom(\calX)^{\circ}$, an isometry $h\in \Isom(\calX)$ can be either holomorphic or antiholomorphic. In the first case it preserves the complex structure $J$ on $\calX$ and the K\"{a}hler form $\omega_{\calD_{\calX}}$, whereas in the second case $\omega_{\calD_{\calX}}$ is sent to $-\omega_{\calD_{\calX}}$. 
When $\textbf{G}$ is the connected adjoint $\matR$-group associated to the complexification of the Lie algebra of $G$ and $s:\textbf{G}\rightarrow\textbf{G}$ is a $\matR$-homomorphism, the induced isometry $h\in \Isom(\calX)$ is holomorphic if and only if $s$ is \emph{positive} and antiholomorphic if it is \emph{negative} \cite[Definition 4.7]{BIW09}. 
In setting of Definition \ref{def_parametrized_kahler_class} the composition of $\sigma$ with $s$ affects the pullback of the K\"{a}hler class by a sign $\pm$, depending on the holomorphicity.
\end{oss}

\section{Proof of the Theorem}\label{section_proof}	

Before starting with the proof of Theorem \ref{main_theorem}, we recall a lemma that holds for cocycles in degree $2$.  

\begin{lemma}{\upshape \cite[Corollary 2.6]{MonShal0}} \label{lem:no:two:coboundary}
Let $\Gamma$ a finitely generated group and  let $(X,\mu_X)$ be a standard Borel probability space. If $B$ is a $\Gamma$-boundary, then 
$$
\Hb^2(\Gamma;\textup{L}^\infty(X;\mathbb{R})) \cong \mathcal{Z}\textup{L}^\infty_{\textup{w}^\ast,\textup{alt}}(B^3; \textup{L}^\infty(X;\mathbb{R}))^\Gamma \ ,
$$
where the letter $\mathcal{Z}$ denotes the set of cocycles and the subscript \emph{alt} denotes the restrictions to alternating essentially bounded weak${}^*$ measurable functions. 
\end{lemma}

We are now ready to give the proof of 
\vspace{0.3 cm}

\begin{rec_thm}[\ref{main_theorem}]
Let $\Gamma$ be a finitely generated discrete group and let $(X,\mu_X)$ be an ergodic standard Borel probability $\Gamma$-space. Let $\{\sigma_i:\Gamma\times X\rightarrow G_i\},i=1,\ldots,n$ be a family of Zariski dense measurable cocycles into simple Hermitian Lie groups not of tube type. If the cocycles are pairwise inequivalent, then the subset 
  $$\{\Hb^2(\sigma_i)(k^b_{G_i}), 1\leq i\leq n\}\subset \Hb^2(\Gamma;\textup{L}^\infty(X;\matR))$$
  is linearly independent over $\textup{L}^\infty(X;\matZ)$.
\end{rec_thm}

\begin{proof}
Suppose the existence of coefficients $m_i\in \Linf(X;\matZ), i=1,\ldots ,n$ such that 
$$\sum\limits_{i=1}^n m_i \Hb^2(\sigma_i)(k^b_{G_i}) =0. $$

Since each cocycle is Zariski dense and $X$ is $\Gamma$-ergodic, Corollary \ref{cor_boundary} guarantees the existence of a boundary map $\phi_i:B\times X\rightarrow \calS_{G_i}$ from a $\Gamma$-boundary $B$ into the Shilov boundary $\calS_{G_i}$ of the group $G_i$. By Example \ref{es_pullback_kahler_boundary}, the cocycle $\textup{C}^2(\Phi_i)(\beta_{G_i})\in \Linf_{\textup{w}^*}(\calS_{G_i}^3;\Linf(X;\matR))$ represents canonically the pullback of $k_{G_i}^b$ along $\sigma_i$. Additionally, since the cocycle is alternating, by Lemma \ref{lem:no:two:coboundary} there are no coboundaries in degree two. 
Hence we get the following equation
\begin{equation}\label{eq1}
\sum\limits_{i=1}^n m_i(x) \beta_{G_i}(\phi_i(b_1,x),\phi_i(b_2,x),\phi_i(b_3,x)) =0 
\end{equation}
that holds for almost every triple $(b_1,b_2,b_3)\in B^3$ and for almost every $x\in X$.
As a consequence of Equation \eqref{eq_relation} it follows that   
\begin{equation}\label{relation}
\prod\limits_{i=1}^n \langle\langle \phi_i(b_1,x),\phi_i(b_2,x),\phi_i(b_3,x)\rangle\rangle_{\matC}^{m_i(x)}=1
\end{equation}
for almost every triple $(b_1,b_2,b_3) \in B^3$ and for almost every $x\in X$.

 For any $i$, Corollary \ref{cor_boundary} allows to choose $\phi_i$ in such a way that the subset of points $(x,b_1,b_2)\in X\times B\times B$ with $(\phi_i(b_1,x),\phi_i(b_2,x))\in \calS_{G_i}^{(2)}$ is of full measure.  Hence, since a finite intersection of full measure sets is still of full measure, we can fix a point $x_0\in X$ and a pair $(b_1,b_2)\in B^2$ such that $(\phi_i(b_1,x_0),\phi_i(b_2,x_0))\in \calS_{G_i}^{(2)}$ for every $i=1,\ldots,n$.
 
Exploiting the transitivity of $G_i$ on pairs in $\calS_{G_i}^{(2)}$, we can identify $\calS_{G_i}^{(2)}$ with the quotient 
$$G_i/\Stab_{G_i}(\phi_i(b_1,x_0),\phi_i(b_2,x_0))$$ by the stabilizer in $G_i$ of the pair $(\phi_i(b_1,x_0),\phi_i(b_2,x_0))\in \calS_{G_i}^{(2)}$. Furthermore, the map $X\rightarrow \calS_{G_i}^{(2)}$ that takes $x$ to the pair $(\phi_i(b_1,x),\phi_i(b_2,x))$ is measurable by the measurability of $\phi_i$. Hence the composition 
$$X\rightarrow \calS_{G_i}^{(2)}\rightarrow G_i/\Stab_{G_i}(\phi_i(b_1,x_0),\phi_i(b_2,x_0))$$ 
is measurable as well and, composing again with the measurable section 
$$G_i/\Stab_{G_i}(\phi_i(b_1,x_0),\phi_i(b_2,x_0))\rightarrow G_i$$ given by \cite[Corollary A.8]{zimmer:libro}, we get a family of measurable functions 
$$g_i:X\rightarrow G_i \ .$$

By setting $\phi_i^{g_i}(b,x)\coloneqq g_i(x)^{-1}\phi_i(b,x)$, we have
\begin{itemize}
 \item $\phi_i^{g_i}(b_1,x)=\phi_i(b_1,x_0)$ for almost every $x\in X$;
 \item $\phi_i^{g_i}(b_2,x)=\phi_i(b_2,x_0)$ for almost every $x\in X$;
\item $\phi_i^{g_i}$ is a boundary map for the cocycle $\sigma_i^{g_i}$ by Remark \ref{oss_twisted_map}. 
\end{itemize}

Thanks to the $G_i$-invariance of the complex Hermitian triple product, Equation \eqref{relation} implies that
\begin{equation}\label{eq3}
\prod\limits_{i=1}^n \langle\langle \phi_i^{g_i}(b_1,x),\phi_i^{g_i}(b_2,x),\phi_i^{g_i}(b_3,x)\rangle\rangle_{\matC}^{m_i(x)}=1
\end{equation}
holds for almost every $b_3 \in B$ and for almost every $x\in X$.
In view of the properties of the $\phi_i^{g_i}$'s, we can rewrite Equation \eqref{eq3} so that 
\begin{equation}\label{eq2}
\prod\limits_{i=1}^n \langle\langle \phi_i(b_1,x_0),\phi_i(   b_2,x_0),\phi_i^{g_i}( b_3,x)\rangle\rangle_{\matC}^{m_i(x)}=1
\end{equation}
holds for almost every $b_3 \in B$ and for almost every $x\in X$.

We define the map
$$\sigma:\Gamma\times X\rightarrow\prod\limits_{i=1}^n G_i,\;\;\;\; (\gamma,x)\mapsto (\sigma_i^{g_i}(\gamma,x))_i$$
which is a cocycle since the cocycle condition follows from the ones of the $\sigma_i^{g_i}$'s
with boundary map
$$\phi:B\times X\rightarrow\prod\limits_{i=1}^n \calS_{G_i},\;\;\;\; (b,x)\mapsto (\phi_i^{g_i}(b,x))_i.$$
and we denote by $\textbf{L}$ the algebraic hull of $\sigma$.

Using the same notation of Lemma \ref{lemma_constant}, we denote by
$\mathcal{O}_i\coloneqq \mathcal{O}_{\phi_i(b_1,x_0),\phi_i(b_2,x_0)} \subset \calS_{G_i}$ the domain of definition of the map 
$$P_{\phi_i(b_1,x_0),\phi_i(b_2,x_0)}:\calO_{\phi_i(b_1,x_0),\phi_i(b_2,x_0)}\rightarrow \Delta^*\backslash A^*\;, \;\;\; \eta\mapsto \langle\langle \phi_i(b_1,x_0),\phi_i(b_2,x_0),\eta\rangle\rangle_{\matC}\,.$$
For almost every $x \in X$ we have that
$$\left( \overline{\textup{EssIm}(\phi_x)}^Z\cap \prod\limits_{i=1}^n\mathcal{O}_i \right) \subset \left\{(\eta_1,\ldots,\eta_n)\in \prod\limits_{i=1}^n\mathcal{O}_i \;,\;  \prod\limits_{i=1}^n P^{m_i(x)}_i(\eta_i)=0\right\}$$
where $\phi_x$ is the $x$-slice of the boundary map. Since each $G_i$ is not of tube type, by Lemma \ref{lemma_constant} each $P_i^{m_i(x)}$ cannot be constant. It follows that $\overline{\textup{EssIm}(\phi_x)}^Z$ is contained in a proper Zariski closed subset of $\prod\limits_{i=1}^n\mathcal{O}_i$. This implies that the cocycle $\sigma$ cannot be Zariski dense, otherwise its slices would be Zariski dense by Corollary \ref{cor_boundary}. As a consequence, we conclude that $\textbf{L}$ must be a proper subgroup of $\prod\limits_{i=1}^n\mathbf{G}_i$.  


Now, since every $\sigma_i$ is Zariski dense, also every $\sigma_i^{g_i}$ is and the projection $\pi_i$ of $\textbf{L}$ on $\mathbf{G}_i$ is onto for every $i$.
In this setting we claim that there exist $i,j\in \{1,\ldots,n\}$ with $i\neq j$ such that $\textbf{G}_i\cong \textbf{G}_j$.
We will sketch a proof for $n=2$ based on Goursat Lemma \cite[Theorem 4]{AC09} the get the statement. The more general case can be deduced from this basic step using the extended version of Goursat Lemma \cite[3.2 Theorem]{Bauer2015}.

Consider a proper subgroup $\textbf{L}<\textbf{G}_1\times \textbf{G}_2$ such that the restriction of the projections $\pi_i$ on the two factors are surjective. 
Let $e_{\mathbf{G}_i} \in \mathbf{G}_i$ be the neutral element of $\mathbf{G}_i$, with $i=1,2$. Define the inclusion $i_1:\mathbf{G}_1 \rightarrow \mathbf{G}_1 \times \mathbf{G}_2$ as $i_1(g_1)=(g_1,e_{\mathbf{G}_2})$ and set $L_1:=i_1^{-1}(\mathbf{L})$. Define similarly the inclusion $i_2:\mathbf{G}_2 \rightarrow \mathbf{G}_1 \times \mathbf{G}_2$ and the group $L_2$. Notice that $L_i$ is a normal subgroup of $\mathbf{G}_i$, for $i=1,2$. 

By the fact that $\mathbf{L}$ surjects on each factor, Goursat Lemma guarantees the existence an isomorphism $s:\mathbf{G}_1/L_1 \rightarrow \mathbf{G}_2/L_2$ defined by $s(g_1L_1)=g_2L_2$ where $(g_1,g_2) \in \mathbf{L}$. We claim that both $L_1$ and $L_2$ are trivial, leading to the desired isomorphism between $\mathbf{G}_1$ and $\mathbf{G}_2$. 

By contradiction, suppose that $L_1$ is not trivial. Since $\mathbf{G}_1$ is simple, we must have $\mathbf{G}_1=L_1$. By the fact that $\varphi$ is an isomorphism, we also have that $\mathbf{G}_2=L_2$. By the way they are defined $L_1$ and $L_2$, we obtain that $\mathbf{L}=\mathbf{G}_1 \times \mathbf{G}_2$, leading to the desired contradiction. Thus both $L_1$ and $L_2$ are trivial and $s:\mathbf{G}_1 \rightarrow \mathbf{G}_2$ is an isomorphism.

In general, there exists at least one $\matR$-isomorphism 
$s:\textbf{G}_i\rightarrow \textbf{G}_j$ for some $i\neq j$ such that $s\circ\sigma_i\simeq \sigma_j$. This contradicts the pairwise inequivalence of the $\sigma_i$'s and concludes the proof.

\end{proof}
 
Theorem \ref{main_theorem} implies the following 
\vspace{0.3 cm}

\begin{rec_thm}[\ref{main_theorem2}]
Let $\Gamma$ be a finitely generated group, let $(X,\mu_X)$ be an ergodic standard Borel probability $\Gamma$-space and consider a Zariski dense measurable cocycle $\sigma:\Gamma\times X\rightarrow G$  into a simple Hermitian Lie group not of tube type.
Then the class $\Hb^2(\sigma)(k_G^b)$ in $\Hb^2(\Gamma;\textup{L}^\infty(X;\matR))$ is non-zero and it determines uniquely the cohomology class of $\sigma$.
\end{rec_thm}

\begin{proof}
The non-vanishing of $\Hb^2(\sigma)(k_G^b)$ is a direct consequence of Theorem \ref{main_theorem2}.
It remains to prove that two cocycles $\sigma_1,\sigma_2:\Gamma\times X\rightarrow G=\Isom(\calX)^{\circ}$ have the same parametrized K\"{a}hler class if and only if they are cohomologous. One direction follows immediately by Proposition \ref{prop_pull_back_invariance}. 

We now prove the other implication. Assuming that
$\Hb^2(\sigma_1)(k_G^b)= \Hb^2(\sigma_2)(k_G^b)$, Theorem \ref{main_theorem2} provides an $\matR$-automorphism $s:\textbf{G}\rightarrow \textbf{G}$ such that 
$s\circ \sigma_1\simeq \sigma_2$, that is 
$s\circ \sigma_1=\sigma_2^f$
for some measurable function $f:X\rightarrow G$. We denote by $h:\calX\rightarrow \calX$ the isometry induced by $s$ of the Hermitian symmetric space $\calX$ such that $G=\Isom(\calX)^{\circ}$.
It is sufficient to prove that $h$ is holomorphic, since $G=\textbf{G}(\matR)^{\circ}= \textup{Hol}(\calX)\cap \textbf{G}(\matR)$, where $\textup{Hol}(\calX)$ is the set of holomorphic automorphisms of $\calX$ (see \cite{satake:1980} or \cite[Proposition 1.7]{milne}). 
Computing the pull back of the bounded K\"{a}hler class of $G$ and exploiting the $G$-invariance, we obtain that
\begin{align*}
\Hb^2(\sigma_2)(k_G^b)&=\Hb^2(\sigma_2^f)(k_G^b)\\
&=\Hb^2(s\circ \sigma_1)(k_G^b)\\
&=\epsilon(h) \Hb^2(\sigma_1)(k_G^b)\\
&=\epsilon(h) \Hb^2(\sigma_2)(k_G^b)
\end{align*}
where $\epsilon(h)$ is the sign of the isometry $h$. We moved from the second to the third line applying Remark \ref{remark_holom} and according to the fact that $h$ is either holomorphic or antiholomorphic. Since $\Hb^2(\sigma_2)(k_G^b)\neq 0$, then $\epsilon(h)=1$ and $h\in \Isom (\calX)^{\circ}=G$ and hence 
$$h \sigma_1 h^{-1}=\sigma_2^f$$ for some element $h\in G$. The thesis follows by setting 
$$\tilde{f}:X\rightarrow G\,,\;\;\; \tilde{f}(x)\coloneqq f(x) h$$
and by the fact that 
$$\sigma_1=\sigma_2^{\tilde{f}}\simeq \sigma_2.$$
\end{proof} 
 
Following \cite{BIW07}, in the setting of Theorem \ref{main_theorem2} we can denote by $\textup{Rep}_{ZD}(\Gamma;G)$ the subset of Zariski dense representations of $\Gamma$ in $G$ modulo conjugation.
By \cite[Theorem 3]{BIW07} the map 
$$K:\textup{Rep}_{ZD}(\Gamma;G)\rightarrow \Hb^2(\Gamma;\matR)\;,\;\;\; [\rho]\mapsto \Hb^2(\rho)(k_G^b)$$
is injective. 
Moreover, the inclusion 
  \begin{align*}  \left\{\parbox{13 em}{Zariski dense representations    \centering$\Gamma\rightarrow G$}\right\}&\hookrightarrow \left\{\parbox{13 em}{Zariski dense cocycles \centering$\Gamma\times X \rightarrow G$ }\right\}\;,\\
  \rho&\mapsto \sigma_{\rho}.
\end{align*}
induces a map 
$$\textup{Rep}_{ZD}(\Gamma;G)\rightarrow \Hm^1_{ZD}(\Gamma \curvearrowright X; G).$$
Finally we denote by
$$K_X: \Hm^1_{ZD}(\Gamma \curvearrowright X;G )\rightarrow\Hb^2(\Gamma;\Linf(X;\matR))\, ,\;\;\; [\sigma]\mapsto  \Hb^2(\sigma)(k_G^b)$$ the map that associates to every cohomology class of a cocycle $\sigma:\Gamma\times X\rightarrow G$ its parametrized K\"{a}hler class. By Theorem \ref{main_theorem} we know that the map $K_X$ is injective. 

Putting together the above maps and the map induced in cohomology by the inclusion of coefficients $\matR\rightarrow \Linf(X;\matR)$, we get the following 

\begin{cor}\label{corollary_diagram}
In the setting of Theorem \ref{main_theorem2}, we have a commutative diagram
\begin{equation}\label{eq:diagram:K}
\xymatrix{
\textup{Rep}_{ZD}(\Gamma;G) \ar[rr]^K \ar[d]  && \Hb^2(\Gamma;\matR) \ar[d] \\
\Hm^1_{ZD}(\Gamma \curvearrowright X;G)\ar[rr]^{K_X} &&\Hb^2(\Gamma;\textup{L}^{\infty}(X;\matR)) \ ,
}
\end{equation}
where the horizontal rows are injective. 
\end{cor}

\begin{oss}
We discuss a little about the vertical maps which appear in Diagram \eqref{eq:diagram:K}. The left vertical one is not necessarily injective, namely two representations inducing cohomologous cocycles are not necessarily conjugated. Similarly, surjectivity does not hold in general, that is not every cohomology class in $\textup{H}^{1}_{ZD}(\Gamma \curvearrowright X;G)$ contains a representation as preferred representative. Surjectivity is exactly the core of \emph{cocycles superrigidity theory} \cite{zimmer:annals,sarti:savini1}.

We now focus our attention on the right vertical arrow, namely the one induced by the inclusion of coefficients $\matR\hookrightarrow \Linf(X;\matR)$. As we are going to see in the next section, under specific assumptions on $\Gamma$, its bounded cohomology with $\Linf(X;\matR)$-coefficients can be equivalently computed by restricting the coefficients to the submodules of $\Gamma$-invariants \cite{Mon10}. A case when this reduction can be done is given by higher rank irreducible lattices. Together with the hypothesis of $\Gamma$-ergodicity on $(X,\mu)$, the reduction of coefficients allows to conclude that the right vertical arrow in Diagram \ref{eq:diagram:K} is actually an isomorphism. 

It is worth noticing that Diagram \eqref{eq:diagram:K} allows to partially translate the study of the left vertical map, which is merely a set-theoretical object, in terms of a homomorphism between bounded cohomology groups.

Finally, we recall the existence of a left (but not right) inverse map for the function $\Hb^2(\Gamma;\matR)\rightarrow \Hb^2(\Gamma;\Linf(X;\matR))$, induced by the integration along $X$. Such a map, called \emph{integration map}, is exploited for instance in the definition of \emph{numerical invariants} for cocycles \cite{sarti:savini1,sarti:savini}. We refer to \cite{moraschini:savini:2} for a more detailed discussion about this topic.
\end{oss}

\section{Consequences of the main theorem}\label{section_consequences}

The aim of this last section is to present some consequences of Theorem \ref{main_theorem2} when $\Gamma$ belongs to specific families of finitely generated groups. 
The second author has recently studied the elementarity properties of cocycles with values into the homeomorphisms of the circle when $\Gamma$ is either a higher rank lattice \cite[Theorem 4]{savini2021} or an irreducible subgroup of a product \cite[Theorem 3]{savini2021}. Here we want to follow the same line,
and we will prove the vanishing of $\Hm^1_{ZD}(\Gamma \curvearrowright X;G)$ under suitable assumptions. The interest in the vanishing of such space is dynamical and comes from theories as measure equivalences or orbit equivalences.

We start with the case of higher rank lattices.
\begin{prop}\label{cor_high_rank}
Let $\Gamma<H=\textup{\textbf{H}}(\matR)^\circ$ be a lattice, where $\textup{\textbf{H}}$ is a connected, simply connected, almost simple $\matR$-group of rank at least two. Let $(X,\mu_X)$ be an ergodic standard Borel probability $\Gamma$-space and consider a simple Hermitian Lie group $G$ not of tube type.
Then $\Hb^2(\Gamma;\matR)\cong 0$ implies that 
$$\left|\Hm^1_{ZD}(\Gamma \curvearrowright X;G)\right|=0\,.$$
\end{prop}

Before giving the proof, which is an application of Theorem \ref{main_theorem2} and of the main results in \cite{Mon10}, we show some examples in which Proposition \ref{cor_high_rank} applies.
\begin{oss}
The vanishing condition assumed in Proposition \ref{cor_high_rank} gives rise to the natural problem of finding family of examples of higher rank lattices $\Gamma$ for which $\Hb^2(\Gamma;\matR)=0$. 

Thanks to Remark \ref{oss bound zariski dense} the injectivity of the comparison map tells us that it is sufficient to verify that $\textup{H}^2(\Gamma;\mathbb{R})$ vanishes. 
Another example is given by $\Gamma<H= \Isom(\calY)$ a torsion free cocompact lattice of a Lie group of rank bigger or equal than $3$ and $\calY$ is not Hermitian symmetric. In fact under this hypothesis \cite[Corollary 1.6]{BM1} implies that $\Hb^2(\Gamma;\matR)=0$. 

\end{oss}

\begin{proof}[Proof of Proposition 4.1]
Theorem \ref{main_theorem2} provides an injective map 
$$K_X:\textup{H}^1_{ZD}(\Gamma \curvearrowright X;G) \rightarrow \Hb^2(\Gamma;\Linf(X;\mathbb{R}))\,,
$$
and shows that $\Hb^2(\sigma)(k_G^b)\neq 0\in \Hb^2(\Gamma;\Linf(X;\matR))$ for any $\sigma\in \Hm_{ZD}^1(\Gamma\curvearrowright X;G)$.
Moreover, since the Banach $G$-module $\Linf(X;\matR)$ is semi-separable, \cite[Corollary 1.6]{Mon10} implies that $$\Hb^2(\Gamma;\Linf(X;\mathbb{R}))\cong \Hb^2(\Gamma;\Linf(X;\mathbb{R})^{\Gamma})\,,$$
where $\Linf(X;\mathbb{R})^{\Gamma}$ denotes the $\Gamma$-invariant vectors of $\Linf(X;\mathbb{R})$. 
Finally, the hypothesis of ergodicity and the vanishing condition on $\Hb^2(\Gamma;\matR)$ show that $$ \Hb^2(\Gamma;\Linf(X;\mathbb{R})^{\Gamma})\cong  \Hb^2(\Gamma;\matR)=0 \ .$$
We conclude by exploiting the fact that $\Hb^2(\sigma)(k_G^b)\neq 0$ whenever $\sigma$ is Zariski dense.
\end{proof}

\begin{oss}\label{oss bound zariski dense}
The setting of Proposition \ref{cor_high_rank} coincides with the one of Zimmer's superrigidity theorem \cite{zimmer:annals}, which can be applied to show that
any Zariski dense cocycle $\sigma : \Gamma\times X\rightarrow G$ is induced by a representation $\rho:\Gamma\rightarrow G$. This means that the right vertical row of Diagram \eqref{eq:diagram:K} is a bijection, namely 
$$\Hm_{ZD}^1(\Gamma \curvearrowright X;G)= \textup{Rep}_{ZD}(\Gamma,G)\,.$$ 
If we consider the map induced in cohomology by the inclusion $\textup{C}_{\textup{cb}}^2(\Gamma;\matR)\hookrightarrow \textup{C}_{c}^2(\Gamma;\matR)$, this is called \emph{comparison map} $\Hb^2(\Gamma;\matR)\rightarrow \Hm^2(\Gamma;\matR)$. By \cite[Theorem 21]{burger2:articolo} the higher rank assumption implies that the comparison map is injective. As a consequence \cite[Corollary 6]{BIW07} provides a bound for the number of Zariski dense cohomology classes of cocycles, precisely
\begin{equation}\label{bound}
\left|\Hm^1_{ZD}(\Gamma \curvearrowright X;G)\right|\leq \dim \Hm^2(\Gamma;\matR)\,.
\end{equation}

In other words, without assuming that $\Hb^2(\Gamma;\matR)$ vanishes, we only have the bound of Equation \eqref{bound}.
\end{oss}


We move now to the case of products, namely when $$\Gamma<H=\prod\limits_{i=1}^n H_i $$ with $n\geq 2$, where each factor $H_i$ is a locally compact and second countable group with $\Hm_{\textup{cb}}^2(H_i;\matR)=0$. We set $$H'_i=\prod\limits_{j\neq i}H_i$$ for $i=1,\ldots,n$
and we assume that each $H'_i$ acts ergodically on $X$ (that is $H$ acts on $X$ \emph{irreducibly} in the sense of Burger-Monod). 
Following \cite{burger2:articolo}, we say that $\Gamma$ is \emph{irreducible} if each projection of $\Gamma$ in $H_i$ is dense in $H_i$. 
In this setting we are able to prove the following result.
\begin{prop}\label{proposition_product}
Let $n\geq 2$. Consider an irreducible lattice $\Gamma<\prod\limits_{i=1}^n H_i$ in a product of locally compact second countable groups with $\Hcb^2(H_i;\matR)=0$ for $i=1,\ldots, n$. and a simple Hermitian Lie group $G$ not of tube type. 
Let $(X,\mu_X)$ be an irreducible standard Borel $H$-space and assume that the $\Gamma$-action is ergodic. Then it holds that 
$$\left|\Hm^1_{ZD}(\Gamma \curvearrowright X;G)\right|=0\,.$$
\end{prop}
\begin{proof}
As in Proposition \ref{cor_high_rank} the inclusion 
$$
\textup{L}^\infty(X;\mathbb{R}) \rightarrow \textup{L}^2(X;\mathbb{R}) 
$$
induces an injective map  
$$
\Hb^2(\Gamma;\textup{L}^\infty(X;\mathbb{R} )) \rightarrow \Hb^2(\Gamma;\textup{L}^2(X;\matR))
$$
by \cite[Corollary 9]{burger2:articolo}. Precomposing with the injection $K_X$, Theorem \ref{main_theorem2} gives back us an inclusion
$$\Hm^1_{ZD}(\Gamma \curvearrowright X;G)\hookrightarrow \Hb^2(\Gamma;\textup{L}^2(X;\mathbb{R})).$$
By \cite[Theorem 16]{burger2:articolo}, we have a decomposition
\begin{equation}\label{eq_sum}
\Hb^2(\Gamma;\textup{L}^2(X;\mathbb{R}))\cong\bigoplus\limits_{i=1}^n \Hcb^2(H_i;\textup{L}^2(X;\mathbb{R})^{H_i'})\cong\bigoplus\limits_{i=1}^n \Hcb^2(H_i;\matR)
\end{equation}
where the latter isomorphism holds thanks the irreducibility of $H$ on $X$. Since each term in the direct sum vanishes and $\Hb^2(\sigma)(k_G^b)$ is non-zero the statement follows. 
\end{proof}
 
As noticed above, Proposition \ref{proposition_product} can be view as the analogous of  \cite[Theorem 3]{savini2021}. There the second author considers the set of \emph{semicohomology} classes of cocycles into $ \textup{Homeo}^+(\mathbb{S}^1)$, where semicohomology is a weaker notion with respect to cohomology.
A deep study of the \emph{parametrized Euler class} of such cocycles, namely the pullback of the \emph{real Bounded Euler class}, allows to show that, under the same assumptions of either Proposition \ref{cor_high_rank} or Proposition \ref{proposition_product}, the cocycles is semicohomologous to a rotational cocycle.

\section*{Acknowledgments} 
We are grateful to the referees for their suggestions and comments, which allowed us to improve the quality of this paper.  

\bibliographystyle{amsalpha}

\bibliography{biblionote}

\end{document}